\documentclass[12pt,a4paper]{article}
\usepackage{amsmath}
\usepackage{amssymb}
\usepackage{amsxtra}
\usepackage{amsthm}
\usepackage{latexsym}
\usepackage[colorlinks]{hyperref}

\usepackage{graphicx}
\usepackage{float}

\usepackage{bbold}
\newcommand{\ch}{\;{\rm ch}}
\newcommand{\sh}{\;{\rm sh}}

\newtheorem{thm}{\rm\bf Theorem}[section]

\newtheorem{prop}[thm]{\rm\bf Proposition}
\newtheorem{cor}[thm]{\rm\bf Corollary}
\newtheorem{rem}[thm]{\rm\bf Remark}

\renewcommand\footnotemark{} 

\begin{document}

\title{\bf Minimum energy exact null-controllability problem for linear time-delay equations
\thanks{The author was supported by the Norwegian Research Council project ''COMAN'' No. 275113.}
}
\author{Pavel Barkhayev}
%
\date{}

\maketitle

\begin{abstract}
We study the minimum energy null-controllability problem for differential equations with point-wise delays.	
For the equations of both neutral and retarded type we reduce the problem of finding the optimal control
to a Volterra integral equation and solve it explicitly.	
We prove that for any initial state and any controllability time the corresponding optimal control belongs to the characteristic space generated by the equation's exponentials.
Besides, we show that the proposed approach can be applied to the systems of retarded equations with one delay term.
\end{abstract}

\section{Introduction}

We consider linear time-delay control equations with point-wise delays:
\begin{equation}\label{eq_intro_equation_neutral}
\dot x(t) + \sum_{k=1}^N d_{k} \dot{x}(t-r_k) = \sum_{k=0}^N a_{k} x(t-r_k) + u(t), \quad t\ge 0, 
\end{equation}
here $x(t)\in \mathbb{R}$ is the state, $u(t)\in \mathbb{R}$ is the control, 
$0=r_0< r_1<\ldots<r_N=1$ are delays, 
$d_k, a_{k}\in \mathbb{R}$ are constant coefficients, 
and we assume that $d_N^2+a_N^2\not=0$. 

It is well known (see e.g.~\cite{Burns_Herdman_Stech_1983}) that the equation (\ref{eq_intro_equation_neutral}) admits a unique continuous solution $x(t)$ for every control function $u\in L^2_{loc}(0,+\infty)$ and every initial  state
$x(t)=x_0(t)$, $t\in [-1,0]$, 
where $x_0\in W^{1,2}(-1,0)$.
Besides, if (\ref{eq_intro_equation_neutral}) is retarded equation, i.e $d_k=0$ for any $k=\overline{1,N}$, then it admits the unique continuous solution $x(t)=x(t; y, x_0, u(t))$ even for initial states of the form
\begin{equation}\label{eq_intro_initial}
\left\{
\begin{array}{l}
x(0)=y,\\  
x(t)=x_0(t),\; t\in [-1,0),
\end{array} 
\right.
\end{equation}
where $(y, x_0)\in \mathbb{R}\times L^2(-1,0) \equiv M^2$ (see e.g.~\cite{Delfour_1980}). Further we study the initial value problem in the form~(\ref{eq_intro_equation_neutral})--(\ref{eq_intro_initial}) assuming that $(y, x_0)\in \widehat{M}^2 \equiv \{(y, x_0): x_0\in W^{1,2}(-1,0), y=x_0(0)\}$ if the equation~(\ref{eq_intro_equation_neutral}) is not retarded.

An initial state $(y, x_0)$ is called null-controllable at time $T$ by means of the equation (\ref{eq_intro_equation_neutral}) if there exists a control $u\in L^2[0,T]$ 
such that  $x(t; y, x_0, u(t))\equiv 0$ for $t\in [T-1, T]$.
We refer to such controls $u(\cdot)$ as admissible from $(y, x_0)$ at time $T$ and denote the set of all admissible controls as $\mathcal{U}_T(y, x_0)$.
The equation (\ref{eq_intro_equation_neutral}) 
is called exactly null-controllable at time $T$ if $\mathcal{U}_T(y, x_0)\not=\emptyset$ for any initial state $(y, x_0)$.

In the present paper we solve the minimum energy optimal control problem
\begin{equation}\label{eq_intro_min_en_pr}
\|u\|_{L^2(0,T)} \rightarrow \min, \quad u\in \mathcal{U}_T(y, x_0),
\end{equation}
and investigate properties of the constructed optimal controls $\widehat{u}_T(t; y,x_0)$.
For this we first construct explicitly the sets of admissible controls $\mathcal{U}_T(y, x_0)$.

It is worth to emphasize that the behavior of solutions of the equation~(\ref{eq_intro_equation_neutral})  may essentially vary  depending on its coefficients.
As we have already mentioned the sets of null-controllable states of retarded equations are wider comparing to equations with delays in derivatives.
Besides, only smooth terminal states may be reached by means of retarded equations, however, if equation is neutral, i.e.  $d_N\not=0$, every state transferable to null can be reached from null as well, i.e.  the problems of exact controllability to and from the null state are equivalent.
Finally, below we compare smoothness of the optimal controls in various cases.

The problem of exact controllability has been investigated in details in many papers. We refer e.g. to \cite{Banks_Jacobs_Langenhop_1975,Jacobs_Langenhop_1976,Yamamoto_1989,Rabah_Sklyar_2007,Rabah_Sklyar_Barkhayev_2016} and references therein for null-controllability analysis of the neutral type systems, and to \cite{Manitius_Triggiani_1978,Salamon_1984,Colonius_1984,Olbrot_Pandolfi_1988} for the systems of retarded type.

If an initial state is null-controllable then there exist many admissible controls.
This naturally leads to the problem of finding the optimal control having the smallest possible energy.
The optimal control problems were investigated in various settings by many authors, we mention e.g.
\cite{Krasovskii_1962,Halanay_1968,Delfour_optimal_1984,Ito_Tarn_1985,Delfour_optimal_1986,Lee_Yung_1996,Boccia_Vinter_2017}.
One of the most popular methods of investigation is Pontryagin's maximum principle which can be applied to a rather general class of problems, but which, often, can give only a partial characterization of the optimal solution.
However, the form of the equation (\ref{eq_intro_equation_neutral}) allows us to construct explicitly the set of all admissible controls and the corresponding trajectories, and then apply more specific methods to solve the  problem~(\ref{eq_intro_min_en_pr}).

First, for any initial state $(y, x_0)$ and $T>1$  we describe the set of all admissible controls $\mathcal{U}_T(y, x_0)$ in $L^2(0,T)$. Every admissible control is determined by some corresponding function in $L^2(0, T-1)$, which we call the control generator, this allows us to rewrite (\ref{eq_intro_min_en_pr}) 
as the problem on a smaller interval. 
We reduce the latter problem to a Volterra integral equation and give its explicit solution by using the Laplace transform method. Smoothness of the solutions depends on the type of the initial equation: $L^{2}$ in the general case, $W^{1,2}$ for retarded equations, and even $W^{2,2}$ for $N=1$.

Properties of solutions of the equation (\ref{eq_intro_equation_neutral}) are determined by its characteristic function
\begin{equation}\label{eq_intro_characteristic}
D(z) = iz {\rm e}^{iz} + \sum_{k=1}^N d_{k} iz {\rm e}^{i(1-r_k) z} - \sum_{k=0}^N a_{k} {\rm e}^{i(1-r_k) z}.
\end{equation}
We denote the zeros of $D(z)$ as $\{z_k\}_{k\in\mathbb{Z}}$ and consider the corresponding system of exponentials $\{ {\rm e}^{iz_k t} \}_{k\in\mathbb{Z}}$ on the interval $[0,T]$.
For simplicity we assume that $D$ does not have multiple zeros.
One can easily see that for $T<1$ this system is infinitely redundant in $L^2(0,T)$, it has excess $1$ in $L^2(0,1)$, and for $T > 1$ it has infinite deficiency; 
and also for $T > 1$ it is minimal (see e.g. \cite{Levin_1996}) in its closure 
\begin{equation}\label{eq_intro_ET}
E_T=\overline{ {\rm Lin}\{{\rm e}^{iz_k t}, {k\in\mathbb{Z}}\}}\subset L^2(0,T).
\end{equation}
In what follows we call $E_T$ the characteristic space of the equation (\ref{eq_intro_equation_neutral}) corresponding to time $T$.
We note that in the case of neutral equations ($d_N\not=0$) the zeros  $\{z_k\}_{k\in\mathbb{Z}}$ belong to a horizontal strip  of the complex plane and the exponentials $\{ {\rm e}^{iz_k t} \}_{k\in\mathbb{Z}}$ might form a Riesz basis of $E_T$, while if $d_N=0$ then the set $\{z_k\}_{k\in\mathbb{Z}}$ belongs to some half-plane $\{z:\: \Im z\ge c\}$ and the set $\{ {\rm e}^{iz_k t} \}_{k\in\mathbb{Z}}$ is only minimal in $E_T$.
This allows us to apply the techniques and methods of non-harmonic Fourier series to study the optimal solutions.

We prove that the optimal controls $\widehat{u}_T(t)$ possess an important feature: for any initial state $(y,x_0)$ the corresponding minimum energy control  belongs to $E_T$ after the symmetric change of time
\begin{equation}\label{eq_intro_charspace}
\widehat{u}_T(T-t; y,x_0)\in E_T.
\end{equation}

Finally, we apply the proposed method to the minimum energy problem for vector time-delay retarded systems with one  delay term
\begin{equation}\label{eq_intro_equation_system}
\dot{\mathbf x}(t)=A {\mathbf x}(t-1) + {\mathbf b} u(t), \quad t\ge 0, 
\end{equation}
here ${\mathbf x}(t) \in \mathbb{R}^n$ is the vector state, $u(t)\in \mathbb{R}$ is the control, $A\in\mathbb{R}^{n \times n}$, ${\mathbf b}\in\mathbb{R}^{n}$. 
The matrix function $D(z)=iz {\rm e}^{iz} I - A$ is the characteristic matrix of the system and we assume that the so-called spectral controllability condition holds:
\begin{equation}\label{eq_intro_equation_spectral_controllavility} 
{\rm rank} (D(z), {\mathbf b}) = n, \quad \mbox{ for any } z\in\mathbb{C}.
\end{equation}
Under this condition the scalar control $u(\cdot)$ applied along the fixed direction $\mathbf b$ actually is redistributed along other directions. This assures null-controllability of the system~(\ref{eq_intro_equation_system}) (see e.g.~\cite{Colonius_1984}).
We are interested in possibility of controlling a vector system by a scalar function. Respectively, it is natural to expect that the null-controllability time cannot be smaller than $n$ in the general settings.
We construct the admissible controls explicitly, what, in particular, gives that an initial state can be null-controllable for any time $T=n+\varepsilon$, $\varepsilon>0$
and, in general, cannot be controllable for the time $T=n$ or smaller.
We show that the minimum energy problem may be reduced to a Volterra integral equation and solve it.

The paper is organized as follows. In Section~2 we give the description of the admissible controls.
In Section~3 we solve the minimum energy problem. In Section~4 we show that the optimal controls belong to the characteristic space $E_T$. To make the explanations more clear, we begin each section with the analysis for the simplest retarded equation with one delay term ($N=1$ and $a_0=0$), then we consider the general equation (\ref{eq_intro_equation_neutral}),
and finally we analyze the case of the vector systems~(\ref{eq_intro_equation_system}).

\section{Admissible controls}

Every admissible control $u\in \mathcal{U}_T(y, x_0)$ first steers the trajectory to zero: $x(T-1; y, x_0, u)=0$ and then keeps it there during the time period of length $1$: $\dot{x}(t; y, x_0, u) = 0$, $t\in[T-1, T]$.

\begin{rem}
	We note that controllability time $T$ for the equations (\ref{eq_intro_equation_neutral}) is greater than $1$. 
	For $T=1$ the set $\mathcal{U}_T(y, x_0)$  is non-empty only if $y=0$ which is not the general case. 
\end{rem}
Similarly the controllability time $T$ for the system~(\ref{eq_intro_equation_system}) is greater than $n$.
\begin{rem}
If a state $(y, x_0)$ is null-controllable at some time $T_0$ then it is null-controllable at any $T>T_0$ (one can put $u(t)=0$, $t\in [T_0, T]$).	
\end{rem}
So further we may assume  for the equations (\ref{eq_intro_equation_neutral}) that $1<T<1+r_1$, where $r_1$ is the smallest delay, and for the systems~(\ref{eq_intro_equation_system}) that $n<T<n+1$.
The form of the equations~(\ref{eq_intro_equation_neutral}) and~(\ref{eq_intro_equation_system})
allows to construct admissible controls explicitly. 
 
In this section we construct the admissible controls for the simplest retarded equation~(\ref{eq_intro_equation_simplest}), then for the general equation~(\ref{eq_intro_equation_neutral}),\
and finally for the vector system~(\ref{eq_intro_equation_system}).

\paragraph{One delay term retarded equations.}

We consider the equation
\begin{equation}\label{eq_intro_equation_simplest} 
\dot{x}(t)=a_1 x(t-1) + u(t), \quad t\ge 0, a_1\in\mathbb{R},
\end{equation}
fix an arbitrary $\varepsilon$: $0<\varepsilon<1$ and $(y, x_0)\in M^2$, and describe the admissible set $\mathcal{U}_T(y, x_0)$, $T=1+\varepsilon$.

For a fixed control $u(t)$ the trajectory $x(t)=x(t; y, x_0, u(t))$ of the initial-value problem (\ref{eq_intro_equation_simplest}),(\ref{eq_intro_initial}) is of the form:
\begin{equation}\label{eq_admissible_trajectory_simplest}
x(t) = \widetilde{x}(t) + \int_0^t  u(\tau) \: {\rm d}\tau, \quad t\in(0,1)
\end{equation}
where $\widetilde{x}(t)$ denotes the trajectory of the equation without input ($u(t)\equiv 0$): 
$$\widetilde{x}(t) = x(t; y, x_0, 0) = y+ a_1 \int_0^t x_0(\tau-1)\: {\rm d}\tau.
$$

The condition $x(T-1)=0$ for the equation (\ref{eq_intro_equation_simplest}) takes the form
\begin{equation}\label{eq_admissible_moments_simplest}
\int_0^\varepsilon u(\tau)\: {\rm d}\tau = - \widetilde{x}(\varepsilon),
\end{equation}
and $\dot{x}(t) = 0$, $t\in[T-1, T]$ is equivalent to $u(t) = - a_1 x(t - 1)$.
This gives us an explicit representation of the admissible controls.
\begin{prop} Every admissible control $u\in \mathcal{U}_T(y, x_0)$, $(y, x_0)\in M^2$ of the equation (\ref{eq_intro_equation_simplest}) is of the form
	\begin{equation}\label{eq_admissible_control_simplest}
	u(t)=\left\{
	\begin{array}{ll}
	u_0(t), & t\in [0,\varepsilon),\\
	- a_1 x_0(t - 1), & t\in [\varepsilon,1),\\
	-a_1 \left( \widetilde{x}(t-1) + \int\limits_0^{t-1} u_0(\tau) \: {\rm d}\tau \right), & t\in [1, 1+\varepsilon),
	\end{array}
	\right.
	\end{equation}
	here the control generator $u_0(\cdot)$ is any function in $L^2(0,\varepsilon)$ satisfying (\ref{eq_admissible_moments_simplest}).
\end{prop}

\begin{rem}
From (\ref{eq_admissible_control_simplest}) it follows that 
$$\mathcal{U}_T(y, x_0)\cap \mathcal{U}_T(\widetilde{y}, \widetilde{x}_0) = \emptyset \quad \mbox{ for any } (y, x_0)\not= (\widetilde{y}, \widetilde{x}_0).$$
\end{rem}

\begin{rem}
The subspace of all admissible controls $\mathcal{U}_T = \mathop{\bigcup}_{(y, x_0)\in M^2} U_T(y, x_0)$ allows the characterization
$$
\mathcal{U}_T = L^2(0,1)\times \hat{W}^{1,2}(1,1+\varepsilon),
$$
where $\hat{W}^{1,2}(1,1+\varepsilon) = \{w\in W^{1,2}(1,1+\varepsilon): \: w(1+\varepsilon) = 0\}$.
\end{rem}
Indeed, any $(v, w)\in L^2\times \hat{W}^{1,2}$ defines $u(t)=\left\{ \begin{array}{ll}
v(t), & t\in [0,1)\\
w(t), & t\in [1, 1+\varepsilon]
\end{array} \right.$ in $\mathcal{U}_T(y, x_0)$ with $(y, x_0)\in M^2$ uniquely determined as
$$
\begin{array}{l}
y = \frac{1}{a_1} u(1),\\
x_0(t) = \left\{ \begin{array}{ll}
- \frac{1}{a_1^2} u'(t+2) - \frac{1}{a_1}  u(t+1), & t\in[-1, -1+\varepsilon),\\
-\frac{1}{a_1}  u(t-1), & t\in[-1+\varepsilon, 0).
\end{array}
\right.
\end{array}
$$

\paragraph{Neutral equations of the general form.}

We fix an arbitrary $\varepsilon$ in the interval $0<\varepsilon<r_1$  and $(y, x_0)\in \hat{M}^2$, and construct the admissible controls $u\in \mathcal{U}_T(y, x_0)$, $T=1+\varepsilon$. 
For a fixed control $u(t)$ and $t\in[0,r_1]$ the trajectory $x(t)=x(t; y, x_0, u(t))$ of the initial-value problem (\ref{eq_intro_equation_neutral})--(\ref{eq_intro_initial}) is of the form:
\begin{equation}\label{eq_admissible_trajectory}
x(t) = \widetilde{x}(t) + \int_0^t {\rm e}^{a_0(t-\tau) } u(\tau) \: {\rm d}\tau,
\end{equation}
where $\widetilde{x}(t)$ is the trajectory of the equation without control ($u(t)\equiv 0$):
\begin{equation}\label{eq_admissible_trajectory_nocontrol}
\widetilde{x}(t) = {\rm e}^{a_0 t}\left(y + \int_0^t {\rm e}^{-a_0 \tau} 
\sum_{k=1}^N [a_k x_0(\tau - r_k) - d_k \dot{x}_0(\tau - r_k)] \: {\rm d}\tau \right).
\end{equation}

From the explicit form (\ref{eq_admissible_trajectory}) of the trajectory 
we get that only smooth initial states  $(y,x_0)\in \widehat{M}^2$ can be null-controllable in case of the neutral equation (\ref{eq_intro_equation_neutral}).
Every admissible control steers the trajectory to zero: $x(\varepsilon)=0$, this may be written as
\begin{equation}\label{eq_admissible_moments_neutral}
\int_0^\varepsilon {\rm e}^{a_0 (t-\tau)}  u(\tau)\: {\rm d}\tau = -\widetilde{x}(\varepsilon).
\end{equation}
The condition $\dot{x}(t) = 0$, $t\in[\varepsilon, 1+\varepsilon]$ is equivalent to 
\begin{equation}\label{eq_admissible_control_neutral}
u(t) = \sum\limits_{k=1}^N \left[d_k \dot{x}(t - r_k)- a_k x(t - r_k)\right] - a_0 x(t), \quad t\in [\varepsilon, 1+ \varepsilon].
\end{equation}
This means that if $t\in [r_{s-1}+\varepsilon,r_s)$, $s=\overline{1,N}$, then $u(t)$ depends on the initial state only:
\begin{equation}\label{eq_admissible_control_neutral_psi}
u(t)=\sum\limits_{k=s}^N [d_k \dot{x}_0(t - r_k) - a_k x_0(t - r_k)]
\equiv \psi_s(t-r_s),
\end{equation}
and if $t\in [r_s, r_s+\varepsilon)$, $s=\overline{1,N}$ then also it depends on the control on the interval $[0,\varepsilon)$:
$$
\begin{array}{rcl}
u(t) & = & d_s \dot{x}(t-r_s) - a_s x(t-r_s) + \sum\limits_{k=s+1}^N[d_k \dot{x}_0(t - r_k) - a_k x_0(t - r_k)] \\
 & = & d_s u_0(t-r_s) + (d_s a_0 - a_s) \int\limits_0^{t-r_s} {\rm e}^{a_0(t -r_s-\tau)} u_0(\tau)\: {\rm d}\tau + \varphi_s(t-r_s),
\end{array}
$$
where 
\begin{equation}\label{eq_admissible_control_neutral_varphi}
\begin{array}{rcl}
\varphi_s(t-r_s) & = & (d_s a_0 - a_s) \widetilde{x}(t-r_s) + \sum\limits_{k=s+1}^N[d_k \dot{x}_0(t - r_k) - a_k x_0(t - r_k)] \\
&  &  -  d_s \sum\limits_{k=1}^N \left[d_k \dot{x}(t -r_s - r_k)- a_k x(t -r_s - r_k)\right].
\end{array}
\end{equation}
Thus we obtain an explicit form of the admissible controls.
\begin{prop} Every admissible control $u\in \mathcal{U}_T(y, x_0)$, $(y, x_0)\in \widehat{M}^2$  of the equation~(\ref{eq_intro_equation_neutral}) is of the form
	\begin{equation}\label{eq_admissible_control2_neutral}
	u(t)=\left\{
	\begin{array}{ll}
	u_0(t), & t\in [0,\varepsilon),\\
	\psi_s(t-r_s), & t\in [r_{s-1}+\varepsilon,r_s), s=\overline{1,N} \\
	\varphi_s(t-r_s) + d_s u_0(t-r_s)  \\
	\quad + (d_s a_0 - a_s) \int\limits_0^{t-r_s} {\rm e}^{a_0(t -r_s-\tau)} u_0(\tau)\: {\rm d}\tau, & t\in [r_s, r_s+\varepsilon), s=\overline{1,N},
	\end{array}
	\right.
	\end{equation}
	here the control generator $u_0(\cdot)$ is any function in $L^2(0,\varepsilon)$ satisfying  (\ref{eq_admissible_moments_neutral}) and the functions $\psi_s(\cdot)$, $\varphi_s(\cdot)$ are given by (\ref{eq_admissible_control_neutral_psi}) and (\ref{eq_admissible_control_neutral_varphi}) respectively.
\end{prop}

\begin{rem}
	For general retarded equations ($d_k=0$, $k=\overline{1,N}$) any initial state $(y,x_0)\in M^2$ is null-controllable and the admissible controls $u\in \mathcal{U}_T(y, x_0)$ are of the form:
	\begin{equation}\label{eq_admissible_control_retarded}
	u(t)=\left\{
	\begin{array}{ll}
	u_0(t), & t\in [0,\varepsilon),\\
	- \sum\limits_{k=s}^N a_k x_0(t - r_k), & t\in [r_{s-1}+\varepsilon,r_s), s=\overline{1,N} \\
	-a_s x(t-r_s) - \sum\limits_{k=s+1}^N a_k x_0(t - r_k), & t\in [r_s, r_s+\varepsilon), s=\overline{1,N},
	\end{array}
	\right.
	\end{equation}
	where $u_0(\cdot)\in L^2(0,\varepsilon)$ satisfies (\ref{eq_admissible_moments_neutral}).
\end{rem}

\paragraph{System of retarded equations.} 

For systems (\ref{eq_intro_equation_system}) null-controllability is equivalent to spectral controllability (\ref{eq_intro_equation_spectral_controllavility}) (see e.g. \cite{Colonius_1984}).
The latter condition is equivalent to controllability of the pair $(A,{\mathbf b})$, 
i.e. to the condition ${\rm rank}({\mathbf b}, A{\mathbf b}, \ldots, A^{n-1}{\mathbf b})=n$. 
This implies existence of a nonsingular space transformation $G\in\mathbb{R}^{n\times n}$ such that
\begin{equation}\label{eq_admissible_change_system}
GAG^{-1} = \left(
\begin{array}{cccc}
-g_1 & \ldots & -g_{n-1}  & -g_n \\
1 & \ldots  & 0 & 0 \\
\vdots & \ddots & \vdots  & \vdots \\
0 & \ldots & 1 & 0 
\end{array}
\right),
\qquad
G{\mathbf b}=
\left(
\begin{array}{c}
1 \\
0 \\
\vdots\\
0 
\end{array}
\right).
\end{equation}
Thus, without loss of generality, we may assume the matrix $A$ and vector ${\mathbf b}$ are of the form (\ref{eq_admissible_change_system}), i.e. (\ref{eq_intro_equation_system}) is of the form
\begin{equation}\label{eq_admissible_changed_system}
\left\{
\begin{array}{rcl}
\dot{x}_1(t) & = & -\sum\limits_{k=1}^N g_k x_k(t-1) + u(t), \\
\dot{x}_2(t) & = & x_1(t-1),\\
 & \ldots & \\
\dot{x}_n(t) & = & x_{n-1}(t-1),
\end{array}
\right.
\qquad
t\ge 0.
\end{equation}

In \cite{Delfour_1980} existence of the unique continuous solution of (\ref{eq_admissible_changed_system})  was shown 
for any control function $u\in L^2_{loc}(0,+\infty)$ and for any initial state
\begin{equation}\label{eq_admissible_initial_state_system}
\left\{
\begin{array}{l}
{\mathbf x}(0) = (y_1,\ldots, y_n)\in \mathbb{R}^n\\
{\mathbf x}(t) = (x^0_1(t), \ldots, x^0_n(t)) \in L^2([-1,0]; \mathbb{R}^n).
\end{array}
\right.
\end{equation}

We note that, in the general settings, any initial state (\ref{eq_admissible_initial_state_system}) cannot be null-controllable for the time $T=n$ or smaller.
Indeed, the component $x_1(t)$ can be null-controllable for any time $T>1$, however $x_2(t)$ is determined by the history $x_1^0(t)$ and $x_2^0(t)$ on the interval $[-1,1]$, thus it cannot be controllable for time less or equal to $2$. Continuing this reasoning we see that the component $x_n(t)$ is determined by the history $x_1^0(t), \ldots, x_n^0(t)$ on the interval $[-1,(n-1)]$, so it cannot be controllable for time less or equal to $n$.

However, for any $\varepsilon>0$ an arbitrary initial state can be null-controllable at time $T=n+\varepsilon$.
Indeed, we can construct the admissible controls as 
\begin{equation}\label{eq_admissible_control2_system}
u(t)=\left\{
\begin{array}{ll}
u_0(t), & t\in [0,\varepsilon),\\
\sum\limits_{k=1}^N g_k x_k(t-1), & t\in [\varepsilon, n+\varepsilon],
\end{array}
\right.
\end{equation}
where $u_0(\cdot)\in L^2(0,\varepsilon)$ satisfies the following moment equalities
\begin{equation}\label{eq_admissible_control2_system2}
x_k(k-1 + \varepsilon) = 0, \quad k=\overline{1,n}.
\end{equation}
Direct computations allow to rewrite  (\ref{eq_admissible_control2_system})--(\ref{eq_admissible_control2_system2}) 
in more convenient form.

\begin{prop} 
Every admissible control of the system~(\ref{eq_admissible_changed_system}) corresponding to an initial state (\ref{eq_admissible_initial_state_system}) is of the form 
\begin{equation}\label{eq_admissible_control_system}
u(t)=\left\{
\begin{array}{ll}
u_0(t), & t\in [0,\varepsilon),\\
\varphi_k(t-k) + \frac{g_k}{(k-1)!} \int\limits_0^{t-k} (t-k - \tau)^{k-1} u_0(\tau)\: {\rm d}\tau, 
& t\in [k, k+\varepsilon], k=\overline{1,n},\\
\psi_k(t-k), & t\in [k-1 +\varepsilon, k],
\end{array}
\right.
\end{equation}
here the control generator $u_0(\cdot)\in L^2(0,\varepsilon)$ satisfies the moment relations
\begin{equation}\label{eq_admis_055}
\int_0^\varepsilon (\varepsilon - \tau)^{k-1} u_0(\tau)\: {\rm d}\tau = c_k^\varepsilon, \quad k=\overline{1,n},
\end{equation}
and the constants $c_k^\varepsilon$ and the functions $\varphi_k(\cdot)$, $\psi_k(\cdot)$ are determined by the initial state~(\ref{eq_admissible_initial_state_system}).
\end{prop}

\begin{rem}\label{rem_form_control}
We note that for both the equation~(\ref{eq_intro_equation_neutral}) and the systems~(\ref{eq_intro_equation_system}) any admissible control $u\in L^2(0, T)$ is determined by a function $u_0(\cdot)$ from $L^2(0, \varepsilon)$ belonging to some affine hyperplane~$H$: 
$$u(t)=Qu_0(t)+f(t), \quad u_0\in H\subset L^2(0, \varepsilon),$$ 
here $f\in L^2(0,T)$, $Q$ is a bounded operator.
\end{rem}


\section{Minimum energy problem}

In this section we solve the minimum energy problem for the simplest retarded equation~(\ref{eq_intro_equation_simplest}), then for the general equation~(\ref{eq_intro_equation_neutral}),\
and finally for the system~(\ref{eq_intro_equation_system}).

According to Remark~\ref{rem_form_control} the optimal control problems may be rewritten in the form $\Phi(u_0)\rightarrow\min$, $u_0\in H$, where $\Phi(\cdot) = \|u\|_{L^2(0,T)}^2$. 
In order to find the optimal solution we use the ideas similar to those developed in \cite{Krein_Nudelman_1975}.
The functional $\Phi: L^2(0,\varepsilon)\rightarrow \mathbb{R}$ is strictly convex, which is inherited by the strict convexity of  norm. So geometrically our aim is to find $a_0>0$ such that the body $\{u_0: \Phi(u_0)\le a_0\}$ is tangent to the hyperplane $H$, and then to find the point of contact.
This point is the desired optimal control $\widehat{u}_0(t)$ and it is unique due to strict convexity of 
$\Phi$.
We show that the latter problem is equivalent to solution of a Volterra integral equation and we solve this equation explicitly by using the Laplace transform method.
Smoothness of the optimal controls varies depending on class of the corresponding equation~(\ref{eq_intro_equation_neutral}).

\paragraph{One delay term retarded equation.}

Since the admissible controls $u\in\mathcal{U}_T(y, x_0)$ for the equation  (\ref{eq_intro_equation_simplest}) are of the form (\ref{eq_admissible_control_simplest}), then
$$
\begin{array}{rcl}
\|u\|_{L^2(0,T)}^2 & = & \|u_0(t)\|_{L^2(0,\varepsilon)}^2
	+ \left\|a_1 x_0(t-1) \right\|_{L^2(\varepsilon, 1)}^2
 	+ \left\|a_1\widetilde{x}(t-1) +a_1 \int_0^{t-r}  u_0(\tau)\: {\rm d}\tau  \right\|_{L^2(1,1+\varepsilon)}^2\\
 & = &  \|u_0(t)\|_{L^2(0,\varepsilon)}^2
 + \left\|a_1\widetilde{x}(t) +a_1 \int_0^t  u_0(\tau)\: {\rm d}\tau  \right\|_{L^2(0,\varepsilon)}^2
 + C,
 \end{array}
$$
where $C= \left\|a_1 x_0(t-1) \right\|_{L^2(\varepsilon, 1)}^2$ does not depend on $u_0$.

This representation means that the optimal control problem (\ref{eq_intro_min_en_pr}) in $L^2(0, T)$ is equivalent to the optimal problem in $L^2(0,\varepsilon)$:
\begin{equation}\label{eq_min_problem_simplest}
\left\{
\begin{array}{l}
\Phi(u_0)\rightarrow \min,\\
\int\limits_0^\varepsilon u_0(\tau)\: {\rm d}\tau = -\widetilde{x}(\varepsilon),
\end{array}
\right.
\end{equation}
where $\Phi(\cdot)\equiv \|u\|_{L^2(0,T)}^2$, $\Phi: L^2(0,\varepsilon)\rightarrow \mathbb{R}$.

\begin{thm}\label{thm_min_simplest}
The problem (\ref{eq_min_problem_simplest}) possesses the unique solution
\begin{equation}\label{eq_min_optimalu_simplest}
\widehat{u}_0(t) = c \ch (a_1t) 
	+ a_1^2 \int_0^t \ch (a_1[t-\tau]) \widetilde{x}(\tau) \: {\rm d}\tau,
\end{equation}
where 
\begin{equation}\label{eq_min_optimalu_simplest_coeff}
c=-a_1(\sh (a_1\varepsilon))^{-1}\left(\widetilde{x}(\varepsilon)  + a_1 \int_0^\varepsilon \sh (a_1[\varepsilon-\tau]) \widetilde{x}(\tau) \: {\rm d}\tau\right).
\end{equation}
\end{thm}
\begin{proof}
The functional $\Phi: L^2(0,\varepsilon)\rightarrow \mathbb{R}$ is strictly convex,
thus, for any $a>0$ the set $B_a=\{u_0: \Phi(u_0)\le a\}$ is convex (if non-empty) and $B_{a_1}\subset B_{a_2}$ as $a_1<a_2$. 
So our aim is to find $a_0$ such that the hyperplane $\{u_0: \int_0^\varepsilon u_0(\tau)\: {\rm d}\tau = c_\varepsilon\}$ is tangent to the body $\{u_0: \Phi(u_0)\le a_0\}$.
The point of contact is the desired optimal control $\widehat{u}_0(t)$ and it is unique due to strict convexity of $\Phi$.
Let $L_{u_0}$ denote the Frech\'et derivative of $\Phi$ at a point $u_0\in L^2(0, \varepsilon)$. 
At the point of contact $\widehat{u}_0$ the equation 
\begin{equation}\label{eq_min_frechet_simplest}
L_{\widehat{u}_0} h =\alpha \langle 1, h(t) \rangle_{L^2(0,\varepsilon)}
\end{equation}
holds for some $\alpha\in\mathbb{R}$ and any $h \in L^2(0,\varepsilon)$.
 
Let us find the Frech\'et derivative from the relation 
$\Phi(u_0 + h) - \Phi(u_0) = L_{{u}_0} h + \overline{o}(\|h\|)$:
$$
\begin{array}{rcl}
\Phi(u_0 + h) - \Phi(u_0) & = & 2\langle u_0(t), h(t) \rangle_{L^2(0,\varepsilon)}  + \|h\|^2 +\\
& & 2 a_1^2 \left\langle \widetilde{x}(t) + \int\limits_0^t u_0(\tau)\: {\rm d}\tau, \int\limits_0^t h(\tau)\: {\rm d}\tau  \right\rangle +   a_1^2 \|\int\limits_0^t h(\tau)\: {\rm d}\tau \|^2.
 \end{array}
 $$
 Since 
 $\left\langle \widetilde{x}(t) + \int_0^t u_0(\tau)\: {\rm d}\tau, \int_0^t h(\tau)\: {\rm d}\tau  \right\rangle_{L^2(0,\varepsilon)} 
 = \int\limits_0^\varepsilon h(t) \int\limits_t^\varepsilon \left(\widetilde{x}(\tau) +  \int\limits_0^\tau u_0(s)\: {\rm d}s \right)\: {\rm d}\tau
 \: {\rm d}t$
and taking into account the inequality $\left\|\int_0^t h(\tau)\: {\rm d}\tau \right\|\le C \|h\|$ which holds for some $C>0$ and any $h \in L^2(0,\varepsilon)$ we obtain
 $$
 L_{\widehat{u}_0} h = 2  \left\langle u_0(t) + a_1^2 \int_t^\varepsilon \int_0^\tau u_0(s)\: {\rm d}s \: {\rm d}\tau +  a_1^2 \int_t^\varepsilon \widetilde{x}(\tau) \: {\rm d}\tau , h(t) \right\rangle 
 $$
 and thus we get that $\widehat{u}_0(t)$ satisfies the integral equation:
\begin{equation}\label{eq_min_thmeq_simplest}
\widehat{u}_0(t) + a_1^2 \int_t^\varepsilon \int_0^\tau \widehat{u}_0(s)\: {\rm d}s \: {\rm d}\tau +  a_1^2 \int_t^\varepsilon \widetilde{x}(\tau) \: {\rm d}\tau = \alpha, \quad t\in [0,\varepsilon],
\end{equation}
where $\alpha$ is an unknown parameter. We notice that
$$
\begin{array}{rcl}
\int\limits_t^\varepsilon \int\limits_0^\tau \widehat{u}_0(s)\: {\rm d}s \: {\rm d}\tau
& = &  \int\limits_0^t \widehat{u}_0(s)\: {\rm d}s  \int\limits_t^\varepsilon \: {\rm d}\tau 
+  \int\limits_t^\varepsilon \widehat{u}_0(s)\: {\rm d}s  \int\limits_s^\varepsilon \: {\rm d}\tau\\
& = & -\int\limits_0^t (t-s) \widehat{u}_0(s)\: {\rm d}s + \int\limits_0^\varepsilon (\varepsilon-s) \widehat{u}_0(s)\: {\rm d}s
 \end{array}
$$
and thus (\ref{eq_min_thmeq_simplest}) may be rewritten as 
\begin{equation}\label{eq_min_thmeq2_simplest}
\widehat{u}_0(t) - a_1^2(t * \widehat{u}_0(t)) =  a_1^2 \int_0^t \widetilde{x}(\tau) \: {\rm d}\tau +c, 
\end{equation}
here 
$c=\alpha-a_1^2 \int\limits_0^\varepsilon (\varepsilon-\tau) \widehat{u}_0(\tau)\: {\rm d}\tau - a_1^2 \int\limits_0^\varepsilon \widetilde{x}(\tau) \: {\rm d}\tau$.
Let us apply the Laplace transform to (\ref{eq_min_thmeq2_simplest}):
$$
\mathcal{L}\{\widehat{u}_0(t) \} = c \frac{s}{s^2 - a_1^2} +  a_1^2 \frac{s}{s^2 - a_1^2} \mathcal{L}\{\widetilde{x}(t)\}.
$$
Applying the inverse Laplace transform we get~(\ref{eq_min_optimalu_simplest}).
Finally, we substitute $\widehat{u}_0(t)$ to the equation of the tangent plane from (\ref{eq_min_problem_simplest}):
$$
c\frac{1}{a_1} \sh(a_1\varepsilon) + a_1\int_0^\varepsilon \sh (a_1[\varepsilon-\tau]) \widetilde{x}(\tau) \: {\rm d}\tau
= -\widetilde{x}(\varepsilon)
$$
and obtain the representation~(\ref{eq_min_optimalu_simplest_coeff}) for the constant $c$.
\end{proof}

\begin{rem}
We note that smoothness of the optimal control varies: 
\begin{itemize}
	\item $\widehat{u}(t)=\widehat{u}_0(t)\in W^{2,2}(0,\varepsilon)$;
	\item $\widehat{u}(t)\in L^{2}(\varepsilon,1)$;
	\item $\widehat{u}(t)\in W^{1,2}(1,1+\varepsilon)$.
\end{itemize}
\end{rem}

From the explicit form of controls~(\ref{eq_admissible_control_simplest}) and  generators of optimal controls (\ref{eq_min_optimalu_simplest}) it follows that for a fixed time $T$ the minimum energy depends linearly on initial states.
However, the dependence of the energy on time is nonlinear.

\begin{rem}
For a fixed initial state $(y,x_0)\in M^2$	the minimum energy increases strictly monotonically as $\varepsilon\rightarrow 0$:
\begin{equation}\label{eq_min_optimalu_simplest_monoton_varepsilon} 
\|\widehat{u}_{1+\varepsilon_1}(t; y, x_0)\|_{L^2(0,1+\varepsilon_1)} >  \|\widehat{u}_{1+\varepsilon_2}(t; y, x_0)\|_{L^2(0,1+\varepsilon_2)}, \qquad \mbox{ as } \varepsilon_1 < \varepsilon_2.
\end{equation}
\end{rem}	
Indeed, the control
$$
\widetilde{u}(t) = \left\{
\begin{array}{ll}
\widehat{u}_{1+\varepsilon_1}(t; y, x_0), & t\in [0,1+\varepsilon_1)\\
0, & t\in [1+\varepsilon_1, 1+\varepsilon_2]
\end{array}
\right.
$$
belongs to $\mathcal{U}_{1+\varepsilon_2}(y, x_0)$ and is not the optimal control in this class due to~(\ref{eq_min_optimalu_simplest}). On the other hand, $\|\widehat{u}_{1+\varepsilon_1}(t; y, x_0)\|_{L^2(0,1+\varepsilon_1)} = \|\widetilde{u}(t)\|_{L^2(0,1+\varepsilon_2)}$.

\begin{rem}
It follows from (\ref{eq_min_optimalu_simplest}) and (\ref{eq_admissible_control_simplest})
that for a fixed initial state $(y,x_0)\not=(0,0)$ the dependence of the minimum energy's growth on $\varepsilon>0$ may be estimated as follows:
\begin{equation}\label{eq_min_optimalu_simplest_estimate_varepsilon}
0<C_0 \le  \|\widehat{u}_{1+\varepsilon}(t; y, x_0)\|_{L^2(0,1+\varepsilon)} \le  \frac{C_1}{\sqrt{\varepsilon}},
\end{equation}
where the constants $C_i$ depends on the initial state.
\end{rem}

\paragraph{Neutral equations of the general form.}

Taking into account the explicit form~(\ref{eq_admissible_control2_neutral}) 
of the admissible controls $u\in\mathcal{U}_T(y, x_0)$ for the general equation (\ref{eq_intro_equation_neutral}) we obtain
$$
\begin{array}{rcl}
\|u\|_{L^2(0,T)}^2 & = & \|u_0(t)\|_{L^2(0,\varepsilon)}^2 + 
\sum\limits_{s=1}^N  \left\|\varphi_s(t) +d_s u_0(t) + (d_s a_0 - a_s) \int\limits_0^t {\rm e}^{a_0(t-\tau)} u_0(\tau)\: {\rm d}\tau  \right\|_{L^2(0,\varepsilon)}^2 \\
& &  + \sum\limits_{s=1}^N  \left\|\psi_s(t) \right\|_{L^2(r_{s-1}+\varepsilon, r_s)}^2,
\end{array}
$$
where the functions $\psi_s(\cdot)$, $\varphi_s(\cdot)$, $s=\overline{1,N}$ determined by the initial state are given by~(\ref{eq_admissible_control_neutral_psi}) and (\ref{eq_admissible_control_neutral_varphi}).

Thus, the optimal control problem (\ref{eq_intro_min_en_pr}) in $L^2(0, T)$ may be equivalently rewritten as the problem in $L^2(0,\varepsilon)$:
\begin{equation}\label{eq_min_problem_neutral}
\left\{
\begin{array}{l}
\Phi(u_0)\rightarrow \min,\\
\int\limits_0^\varepsilon {\rm e}^{a_0 (\varepsilon-\tau)} u_0(\tau)\: {\rm d}\tau = -\widetilde{x}(\varepsilon),
\end{array}
\right.
\end{equation}
where $\Phi(\cdot)\equiv \|u\|_{L^2(0,T)}^2$, $\Phi: L^2(0,\varepsilon)\rightarrow \mathbb{R}$ and $\widetilde{x}(\cdot)$
is defined by (\ref{eq_admissible_trajectory_nocontrol}).

\begin{thm}\label{thm_min_neutral}
The problem (\ref{eq_min_problem_neutral}) possesses the unique solution
	\begin{equation}\label{eq_min_controlu_neutral}
	\begin{array}{rcl}
	\widehat{u}_0(t) & = & d^{-2}\left[
	\sum\limits_{s=1}^N \int\limits_0^t \left((d_s a_0 - a_s)\ch (\hat{a}[t-\tau])+ 
	\left(\frac{a_sa_0}{\hat{a}}  - d_s\hat{a}\right) \sh (\hat{a}[t-\tau])\right) \varphi_s(\tau) \: {\rm d}\tau
	\right.\\
	& & \left. - \sum\limits_{s=1}^N d_s \varphi_s(t) + c \left(\ch (\hat{a}t) - \frac{a_0}{\hat{a}} \sh (\hat{a}t) \right)\right],
	\end{array}
	\end{equation}
	where $d^2 = 1+ \sum_{k=1}^N d_k^2$, $\hat{a}^2=d^{-2}\sum_{k=0}^N a_k^2$, and the constant $c$ determined from the relation~(\ref{eq_admissible_moments_neutral}) is of the form
	\begin{equation}\label{eq_min_controlu_neutral_coeff}
	c=\frac{1}{\sh (\hat{a}\varepsilon)}\left(-d^2 \hat{a} \widetilde{x}(\varepsilon)  + \int_0^\varepsilon \sum_{s=1}^N \left(\hat{a} d_s \ch (\hat{a} [\varepsilon-\tau]) + a_s \sh (\hat{a} [\varepsilon-\tau])\right) \varphi_s(\tau) \: {\rm d}\tau\right).
	\end{equation}
\end{thm}
\begin{proof}	
	The idea of the proof is similar to Theorem~\ref{thm_min_simplest} and here we single out the differences.
	The functional $\Phi: L^2(0,\varepsilon)\rightarrow \mathbb{R}$ is strictly convex and direct computations gives us its Frech\'et derivative form the relation 
	$\Phi(u_0 + h) - \Phi(u_0) = L_{{u}_0} h + \overline{o}(\|h\|)$:
	$$
	\begin{array}{rcl}
	L_{{u}_0} & = & \left(1+\sum\limits_{k=1}^N d_k^2 \right) u_0(t) 
	+ \sum\limits_{k=1}^N d_k(d_k a_0 - a_k)\int\limits_0^t {\rm e}^{a_0 (t-\tau)}u_0(\tau)\: {\rm d}\tau
	+ \sum\limits_{k=1}^N d_k \varphi_k(t)\\
	&  & + \sum\limits_{k=1}^N d_k(d_k a_0 - a_k) \int\limits_t^\varepsilon {\rm e}^{-a_0 (t-\tau)}u_0(\tau)\: {\rm d}\tau
	+ \sum\limits_{k=1}^N (d_k a_0 - a_k) \int\limits_t^\varepsilon {\rm e}^{-a_0 (t-\tau)}\varphi_k(\tau)\: {\rm d}\tau \\
	&  & + \sum\limits_{k=1}^N   (d_k a_0 - a_k)^2 \int\limits_t^\varepsilon {\rm e}^{-a_0 (t-\tau)}  \int\limits_0^\tau {\rm e}^{a_0(\tau-s)}  u_0(s)\: {\rm d}s \: {\rm d}\tau.
	\end{array}
	$$
	
	 Since the tangent hyperplane is given by $\{u_0: \int_0^\varepsilon {\rm e}^{a_0 (t-\tau)} u_0(\tau)\: {\rm d}\tau = -\widetilde{x}(\varepsilon)\}$ we obtain the integral equation for the optimal control $\widehat{u}_0(t)$:
	\begin{equation}\label{eq_minl_thmeq_neutral0}
	 L_{\widehat{u}_0} = \alpha {\rm e}^{-a_0 t},
	\end{equation}
	where $\alpha$ is an unknown constant. By using the relation
	$$
	{\rm e}^{-a_0 t} \int\limits_t^\varepsilon {\rm e}^{2a_0 \tau} \int\limits_0^\tau {\rm e}^{-s a_0}  \widehat{u}_0(s)\: {\rm d}s \: {\rm d}\tau
	$$
	$$
	= \frac{1}{2a_0} \int\limits_0^t ({\rm e}^{-a_0 (t-s)} - {\rm e}^{a_0 (t-s)}) \widehat{u}_0(s)\: {\rm d}s 
	+ \frac{{\rm e}^{-a_0 t} }{2a_0}  \int\limits_0^\varepsilon ({\rm e}^{2a_0 \varepsilon} - {\rm e}^{2 a_0 s}) {\rm e}^{- a_0 s}  \widehat{u}_0(s)\: {\rm d}s,
	$$
	we rewrite~(\ref{eq_minl_thmeq_neutral0}) as 
	$$
	\left(1+\sum_{k=1}^N d_k^2 \right) \widehat{u}_0(t) + 
	\frac{\sum_{k=1}^N (a_k^2 - d_k^2 a_0^2)}{2a_0} ({\rm e}^{-a_0 t} - {\rm e}^{a_0 t}) * \widehat{u}_0(t)
	$$ 
	\begin{equation}\label{eq_minl_thmeq_neutral}
	=  
	\sum_{k=1}^N (d_k a_0 - a_k)({\rm e}^{-a_0 t} * \varphi_k(t)) - \sum_{k=1}^N d_k \varphi_k(t) + c {\rm e}^{-a_0 t}, 
	\end{equation}
	where $c$ is an unknown constant.
	Further, we apply the Laplace transform to (\ref{eq_minl_thmeq_neutral}):
	$$
	d^2 \mathcal{L}\{\widehat{u}_0(t) \} \left(1 +\frac{\sum_{k=1}^N (a_k^2 - d_k^2 a_0^2)}{2a_0d^2} \left(\frac{1}{s+a_0} -\frac{1}{s-a_0} \right) \right)  
	$$
	$$
	= \frac{1}{s+a_0} \sum_{k=1}^N (d_k a_0 - a_k) \mathcal{L}\{\varphi_k(t)\} - 
	\sum_{k=1}^N d_k \mathcal{L}\{\varphi_k(t)\} +c\frac{1}{s+a_0} ,
	$$
	this gives us
	$$
	d^2 \mathcal{L}\{\widehat{u}_0(t) \} = 
	\frac{s - a_0}{s^2 - \hat{a}^2} \sum_{k=1}^N (d_k a_0 - a_k) \mathcal{L}\{\varphi_k(t)\} 
	+  \sum_{k=1}^N d_k \left(1+  \frac{\hat{a}^2 - a_0^2}{s^2 - \hat{a}^2}  \right) \mathcal{L}\{\varphi_k(t)\}
	+ c \frac{s - a_0}{s^2 - \hat{a}^2}.
	$$
	Applying the inverse Laplace transform we get (\ref{eq_min_controlu_neutral})
	and substituting $\widehat{u}_0(t)$ to the equation of the tangent plane~(\ref{eq_admissible_moments_neutral})
	and integrating by parts we obtain the representation~(\ref{eq_min_controlu_neutral_coeff}) for the constant $c$.
\end{proof}

\begin{rem}
	We note that, in general, $\widehat{u}_0(t)\in L^2(0,\varepsilon)$ since its representation~(\ref{eq_min_controlu_neutral}) includes $\varphi_s(\cdot)\in L^2(0,\varepsilon)$. Moreover, from~(\ref{eq_admissible_control2_neutral}) we conclude that the corresponding minimum energy control $\widehat{u}_T(t; y, x_0)\in L^2(0,T)$.
\end{rem}

\begin{cor}\label{thm_min_retarded}
	For the general retarded equations ($d_k=0$, $k=\overline{1,N}$) the problem (\ref{eq_min_problem_neutral}) possesses the unique solution
	\begin{equation}\label{eq_min_controlu_retarded}
	\widehat{u}_0(t) = c \left(\ch (\hat{a}t) - \frac{a_0}{\hat{a}} \sh (\hat{a}t) \right) 
	+ \sum_{s=1}^N a_s \int_0^t \left(\frac{a_0}{\hat{a}} \sh (\hat{a}[t-\tau]) - \ch (\hat{a}[t-\tau])\right) \varphi_s(\tau) \: {\rm d}\tau,
	\end{equation}
	where $\hat{a}^2=\sum_{k=0}^N a_k^2$, and 
	\begin{equation}\label{eq_min_constantc_retarded}
	c=(\sh (\hat{a}\varepsilon))^{-1} \left(-\widetilde{x}(\varepsilon) \hat{a}  
	+ \sum\limits_{s=1}^N a_s \int_0^\varepsilon \sh (\hat{a}[\varepsilon-\tau]) \varphi_s(\tau) \: {\rm d}\tau\right).
	\end{equation}
\end{cor}

\begin{rem}
	For general retarded systems $\widehat{u}_0(t)\in W^{1,2}(0,\varepsilon)$,
	and for $t\in[\varepsilon, 1+\varepsilon)$: $\widehat{u}_T(t; y, x_0)\in L^{2}(\varepsilon,1+\varepsilon)$.
\end{rem}

\paragraph{System of retarded equations.}

From the representation~(\ref{eq_admissible_control_system}) of admissible controls  we get that
$$
\begin{array}{rcl}
\Phi(u_0)  & = & \|u\|_{L^2(0,T)}^2  \\
& = &  \|u_0\|_{L^2(0,\varepsilon)}^2 + 
\sum\limits_{k=1}^N  \left\|\varphi_k(t) +\frac{g_k}{(k-1)!} t^{k-1}*u_0(t) \right\|_{L^2(0,\varepsilon)}^2
+ \sum\limits_{s=1}^N  \left\|\psi_s(t) \right\|^2,
\end{array}
$$
where the functions $\psi_s(\cdot)$, $\varphi_s(\cdot)$ are determined by the initial state.
So, taking into account (\ref{eq_admis_055}) we obtain the optimal control problem
\begin{equation}\label{eq_min_problem_system}
\left\{
\begin{array}{l}
\Phi(u_0)\rightarrow \min,\\
\int\limits_0^\varepsilon (\varepsilon - \tau)^{k-1} u_0(\tau)\: {\rm d}\tau = c_k^\varepsilon, \quad k=\overline{1,n}.
\end{array}
\right.
\end{equation}
Geometrically our aim is to find $a_0>0$ such that the body $\{u_0: \Phi(u_0)\le a_0\}\subset L^2(0,\varepsilon)$ is tangent to the intersection of the $n$ hyperplanes defined by (\ref{eq_admis_055}).
This means that at the point of contact $\widehat{u}_0(t)$ we will have
\begin{equation}\label{eq_min_frechet_system}
\ell_{\widehat{u}_0} = \sum_{k=1}^n q_k (\varepsilon - t)^{k-1},
\end{equation}
for some constants $\{q_k\}$, where  $\ell_{\widehat{u}_0}$ is the Frech\'et derivative of $\Phi$.

\begin{thm}\label{thm_min_system}
	The problem (\ref{eq_min_problem_system}) possesses the unique solution
	\begin{equation}\label{eq_min_controlu_system}
	\widehat{u}_0(t) = \mathcal{L}^{-1}\left( \frac{s^{2n}}{s^{2n} +\sum_{k=1}^n (-1)^k g_k^2 s^{2(n-k)}} \right)
	* \left(\sum_{k=1}^n q_k (\varepsilon - t)^{k-1}
	+ \sum\limits_{k=1}^N \frac{(-1)^k  g_k}{(k-1)!} t^{k-1}*\varphi_k(t) \right),
	\end{equation}
	where $\mathcal{L}^{-1}$ is the inverse Laplace transform and the constants $q_k$, $k=\overline{1,n}$ can be found from the system ol linear algebraic equations (\ref{eq_admis_055}).
\end{thm}
\begin{proof}
We denote $\Phi_k(u_0)=\left\|\varphi_k(t) +\frac{g_k}{(k-1)!} t^{k-1}*u_0(t) \right\|_{L^2(0,\varepsilon)}^2$ and compute its Frech\'et derivative $\ell_{\widehat{u}_0}^k$: 
$$
\Phi_k(\widehat{u}_0 + h) - \Phi_k(\widehat{u}_0) 
$$
$$
= 2  \frac{g_k}{(k-1)!} \int_0^\varepsilon \left[\varphi_k(t) + \frac{g_k}{(k-1)!} t^{k-1}*\widehat{u}_0(t)\right]\cdot t^{k-1}*h(t)\: {\rm d}t + \overline{o}(\|h(t)\|).
$$	
We integrate by parts $k$ times and obtain that
$$
\ell_{\widehat{u}_0}^k = 2g_k \int_t^\varepsilon \int_{\tau_1}^\varepsilon \ldots \int_{\tau_{k-1}}^\varepsilon
\left[\varphi_k(\tau_k) + \frac{g_k}{(k-1)!} \tau_k^{k-1}*\widehat{u}_0(t)\right] \: {\rm d}\tau_k \ldots {\rm d}\tau_1.
$$	
Taking into account $\int_{\tau_{s}}^\varepsilon = \int_0^\varepsilon - \int_0^{\tau_{s}}$ and adding powers in convolution product we get:
$$
\ell_{\widehat{u}_0}^k = 2\left[\frac{(-1)^k g_k}{(k-1)!} t^{k-1}*\varphi_k(t) + \frac{(-1)^k g_k^2}{(2k-1)!} t^{2k-1}* \widehat{u}_0(t) + \sum_{s=1}^k \widetilde{q}_{k,s-1}(\varepsilon - t)^{s-1}\right],
$$	
where $\widetilde{q}_{k,s-1}$ are constants. Thus we can rewrite (\ref{eq_min_frechet_system}) as	
$$
\widehat{u}_0(t) + \sum_{k=1}^N \frac{(-1)^k g_k^2}{(2k-1)!}t^{2k-1}* \widehat{u}_0(t)
=
\sum_{k=1}^N \frac{(-1)^k g_k}{(k-1)!} t^{k-1}*\varphi_k(t) + \sum_{k=1}^n \widehat{q}_k (\varepsilon - t)^{k-1}.
$$
We apply the Laplace transform to the latter equation:
$$
\mathcal{L}\{\widehat{u}_0(t)\} \left(1+\sum_{k=1}^N \frac{(-1)^k g_k^2}{s^{2k}}
\right)
= \mathcal{L}\left\{ \sum_{k=1}^N \frac{(-1)^k g_k}{(k-1)!} t^{k-1}*\varphi_k(t) + \sum_{k=1}^n \widehat{q}_k (\varepsilon - t)^{k-1}  \right\}
$$
and after a simple transformation and 
applying the inverse Laplace transform we get (\ref{eq_min_controlu_system}).
\end{proof}


\section{The characteristic space and optimal controls}

In this section we consider the solutions $\widehat{u}_T(t)$ of the  minimum energy problem for the equation (\ref{eq_intro_equation_neutral}) and prove that after the  symmetric change of time they belong to the characteristic subspace: $\widehat{u}_T(T-t;y,x_0)\in E_T$, for any initial state $(y,x_0)$ and any $T>1$. 
For this we describe the orthogonal complement $(E_T)^\perp$ and show that optimal controls are orthogonal to $(E_T)^\perp$.
We note that such property holds for some other dynamical systems, e.g. for non-homogeneous vibrating string  \cite{Sklyar_Szkibiel_20012}, rotating beam \cite{Sklyar_Wozniak_2015}, etc. 
To illustrate the ideas of the proof we first consider the case of the simplest retarded equation~(\ref{eq_intro_equation_simplest}) and then we consider the general equation~(\ref{eq_intro_equation_neutral}).

\paragraph{One delay term retarded equations.}

The characteristic function~(\ref{eq_intro_characteristic}) of the equation (\ref{eq_intro_equation_simplest}) is  $D(z) = a_1 - iz {\rm e}^{i z r}$. We denote the set of its zeros by $\Lambda=\{z_k\}_{k\in\mathbb{Z}}$
and the corresponding family of exponentials by $\mathcal{E}(\Lambda)=\left\{{\rm e}^{i z_k t} \right\}_{k\in\mathbb{Z}}.$
For any $\varepsilon>0$, this system is minimal and of infinite deficiency in the space $L^2(0,T)$, $T=1+\varepsilon$ (see e.g. \cite{Levin_1996}). 
We also note that $x(t; 1, {\rm e}^{i z_k t}, 0) = {\rm e}^{i z_k t}$.
By $E_T$ we denote the closure in $L^2(0, T)$ of its linear span:
$$
E_T = \overline{{\rm Lin}\: \mathcal{E}(\Lambda)} \subset L^2(0, T).
$$
We begin with characterization of the orthogonal complement $(E_T)^\perp$.

\begin{prop}\label{prop_characteristic_simplest}
Function $f(t)\in (E_T)^\perp$ if and only if there exists $q(t)\in {W}^{1,2}(0,\varepsilon)$,  $q(0)=q(\varepsilon)=0$ such that 
\begin{equation}\label{eq_characteristic_eorth_simplest}
f(t) = \left\{
\begin{array}{rl}
a_1 q(t), & t\in [0,\varepsilon],\\
0, & t\in [\varepsilon, 1],\\
q'(t-1), & t\in [1,1+\varepsilon].
\end{array}
\right. 
\end{equation}
\end{prop}
\begin{proof}
Let $PW_{(a,b)}$ denotes the Fourier transform of $L^2(a,b)$ (the Paley-Wiener space of entire functions, see e.g. \cite{Levin_1996}). We denote $F(z)=\int_{0}^T{\rm e}^{i z t}  f(t) \: {\rm d} t \in PW_{(0,T)}$.
First let us show that $f(t)\in (E_T)^\perp$ if and only if there exists $Q(z)\in PW_{(0,\varepsilon)}$ such that $zQ(z)\in PW_{(0,\varepsilon)}$ and 
$$
F(z)=D(z)Q(z).
$$
Indeed, by construction $D(z)Q(z) \in PW_{(0,T)}$, thus by the Paley-Wiener theorem there exists $f(t)\in L^2(0,T)$ such that 
 $D(z)Q(z)=\int_{0}^T  {\rm e}^{i z t} f(t) \: {\rm d} t \in PW_{(0,T)}$. Since $D(z_k)=0$, for any $k\in \mathbb{Z}$ then $f(t)\in (E_T)^\perp$. 
Inversely, if $f(t)\in (E_T)^\perp$ then $F(z_k)=0$ for any $k\in \mathbb{Z}$. Then $Q(z)\equiv\frac{F(z)}{D(z)}$ is entire function.
Due to the form of $D(z)$ one has $\frac{F(x)}{D(x)}\in L^2(\mathbb{R})$ and  $\frac{xF(x)}{D(x)}\in L^2(\mathbb{R})$ since $\frac{x}{D(x)}$ is bounded along the real axis.  Thus, $Q(z)$ and $zQ(z)$ belong to $PW_{(0,\varepsilon)}$.

Further, by the Paley-Wiener theorem there exists $q(t)\in L^2(0, \varepsilon)$ such that $Q(z)=\int_{0}^\varepsilon {\rm e}^{i z t} q(t) \: {\rm d} t$. Function $q(t)$ admits the characterization
$$
q(t)\in {W}^{1,2}(0,\varepsilon), \qquad q(0)=q(\varepsilon)=0.
$$
Indeed, since $zQ(z) \in PW_{0,\varepsilon}$, then $zQ(z)=\int_{0}^\varepsilon {\rm e}^{i z t} g(t) \: {\rm d} t$. When $z=0$ we get 
$\int_{0}^\varepsilon g(t) \: {\rm d} t = 0$ and thus
$$
\int_{0}^\varepsilon {\rm e}^{i z t} g(t) \: {\rm d} t = \int_{0}^\varepsilon {\rm e}^{i z t} \left(\int_0^t g(\tau ) \: {\rm d} \tau \right)'  \: {\rm d} t = - iz  \int_{0}^\varepsilon {\rm e}^{i z t} \int_0^t g(\tau ) \: {\rm d} \tau  \: {\rm d} t.
$$
We obtain that $q(t) =  - i \int_0^t g(\tau ) \: {\rm d} \tau$ and thus $q(t)\in {W}^{1,2}(0,\varepsilon)$ and $q(0)=q(\varepsilon)=0$.
We have also shown that $iz Q(z) =- \int_{0}^\varepsilon {\rm e}^{i z t} q'(t) \: {\rm d} t$.

Finally we apply the inverse Fourier transform to $D(z)Q(z)$:
$$
f(t)=\frac{1}{2\pi}\int_{-\infty}^{\infty} {\rm e}^{-i \xi t} D(\xi) Q(\xi) {\rm d} \xi = 
\frac{1}{2\pi}\int_{-\infty}^{\infty} {\rm e}^{-i \xi t} a_1Q(\xi) {\rm d} \xi - \frac{1}{2\pi}\int_{-\infty}^{\infty} {\rm e}^{-i \xi (t-1)} i\xi Q(\xi) {\rm d} \xi 
$$
and this gives us the representation (\ref{eq_characteristic_eorth_simplest}).
\end{proof}

\begin{thm}\label{lem_characteristic_simplest}
For any initial condition $(y, x_0)\in M^2$ and any $T>1$ the minimum energy control of the equation (\ref{eq_intro_equation_simplest}) belongs to the characteristic space:
 $$\widehat{u}_T(T-t; y,x_0)\in E_T.$$
\end{thm}
\begin{proof}
For a given $(y, x_0)\in M^2$ the optimal control is of the form (\ref{eq_admissible_control_simplest}) where $\widehat{u}_0(t)$ is given by (\ref{eq_min_optimalu_simplest}).
Thus we obtain
$$
\widehat{u}_T(T-t; y,x_0)=\left\{
\begin{array}{ll}
-a_1 \widetilde{x}(\varepsilon-t) - a_1 \int\limits_0^{\varepsilon - t} \widehat{u}_0(\tau) \: {\rm d}\tau, & t\in [0,\varepsilon),\\
-a_1 x_0(\varepsilon - t), & t\in [\varepsilon,1),\\
\widehat{u}_0(1+\varepsilon - t), & t\in [1, 1+\varepsilon).
\end{array}
\right.
$$
For an arbitrary $f(t)\in (E_T)^\perp$, taking into account Proposition \ref{prop_characteristic_simplest}, we have:
$$
\begin{array}{l}
\langle \widehat{u}_T(T-t; y,x_0), f(t) \rangle_{L^2(0,T)}   \\
 =-\int\limits_0^\varepsilon \left( a_1\widetilde{x}(\varepsilon-t) + a_1 \int\limits_0^{\varepsilon-t} \widehat{u}_0(\tau) \: {\rm d}\tau \right) a_1 q(t)\:{\rm d} t 
 +\int\limits_r^{r+\varepsilon} \widehat{u}_0(r+\varepsilon - t) q'(t-r)\:{\rm d} t\\
 =\int\limits_0^\varepsilon \left( \int\limits_t^\varepsilon \left( \widetilde{x}(\tau) 
 + \int\limits_0^{\tau} \widehat{u}_0(s) \: {\rm d}s \right) \: {\rm d}\tau \right)' 
  a_1^2 q(\varepsilon-t)\:{\rm d} t + 
 \int\limits_0^{\varepsilon} \widehat{u}_0(t) q'(\varepsilon - t)\:{\rm d} t\\
=\int\limits_0^\varepsilon \left(  \widehat{u}_0(s) 
+ a_1^2  \int\limits_t^\varepsilon \int\limits_0^{\tau} \widehat{u}_0(s) \: {\rm d}s \: {\rm d}\tau 
+ a_1^2 \int\limits_t^\varepsilon \widetilde{x}(\tau) \: {\rm d}\tau\right)
q'(\varepsilon-t)\:{\rm d} t .
\end{array}
$$
The function $\widehat{u}_0(t)$ satisfies the integral equation (\ref{eq_min_thmeq_simplest}), thus
$$
\left\langle \widehat{u}_T(T-t; y,x_0), f(t) \right\rangle_{L^2(0,T)} = 
\int_0^\varepsilon \alpha q'(\varepsilon - t)\:{\rm d} t = 0.
$$
\end{proof}

\begin{rem} 
The optimal control has the structure:
$$
\widehat{u}_T(T-t)  = \left\{
\begin{array}{rl}
\frac{1}{a_1} w'(t), & t\in [0,\varepsilon],\\
*, & t\in [\varepsilon, 1],\\
w(t-1), & t\in [1,1+\varepsilon],
\end{array}
\right. 
$$
where $w(t) = \widehat{u}_0(\varepsilon - t)$.
\end{rem}

\begin{rem} 
$g(t)\in E_T$ if and only if there exists $w(t)\in {W}^{1,2}(0,\varepsilon)$ such that 
\begin{equation}\label{eq_characteristic_remeq_simplest}
 g(t) = \left\{
\begin{array}{rl}
\frac{1}{a_1} w'(t), & t\in [0,\varepsilon],\\
*, & t\in [\varepsilon, 1],\\
w(t-1), & t\in [1,1+\varepsilon].
\end{array}
\right.
\end{equation}
\end{rem}
Indeed, if $w_k(t) = {\rm e}^{i z_k (t+1)}$ we have $\frac{1}{a_1} w_k'(t) = {\rm e}^{i z_k t}$ and $w_k(t-1) = {\rm e}^{i z_k t}$,
thus ${\rm e}^{i z_k t}$ is of the form (\ref{eq_characteristic_remeq_simplest}).
On the other hand any function of the form (\ref{eq_characteristic_remeq_simplest}) is orthogonal to $(E_T)^\perp$.

\paragraph{Neutral equations of the general form.}

The characteristic function $D(z)$ of the equation (\ref{eq_intro_equation_neutral}) is given by (\ref{eq_intro_characteristic}). 
We assume for simplicity that $D$ does not have multiple roots.
The system of exponentials $\left\{{\rm e}^{i z_k t} \right\}_{k\in\mathbb{Z}}$, constructed by the zeros $\{z_k\}_{k\in\mathbb{Z}}$ of $D(z)$,
is minimal and of infinite deficiency in the space $L^2(0,T)$, $T=1+\varepsilon$.
We note that  $x(t; 1, {\rm e}^{i z_k t}, 0) = {\rm e}^{i z_k t}$.
The subspace $E_T$ given by~(\ref{eq_intro_ET}) is the closure of the linear span of the exponentials.
The orthogonal complement $(E_T)^\perp$ allows the following characterization.

\begin{prop}\label{prop_characteristic_neutral}
Function $f(t)\in (E_T)^\perp$ if and only if there exists $q(t)\in {W}^{1,2}(0,\varepsilon)$,  $q(0)=q(\varepsilon)=0$ such that 
\begin{equation}\label{eq_characteristic_eorth_retarded}
f(t) = \left\{
\begin{array}{rl}
a_N q(t) + d_N q'(t), & t\in [0,\varepsilon]\\
0, & t\in [\varepsilon, 1-r_{N-1}]\\
a_{N-1} q(t-1+r_{N-1}) + d_{N-1} q'(t-1+r_{N-1}), & t\in [1-r_{N-1},1-r_{N-1}+\varepsilon]\\
0, & t\in [1-r_{N-1}+\varepsilon, 1-r_{N-2}]\\
\ldots & \\
a_{1} q(t-1+r_{1}) + d_{1} q'(t-1+r_{1}), & t\in [1-r_{1},1-r_{1}+\varepsilon]\\
0, & t\in [1-r_{1}+\varepsilon, 1]\\
a_{0} q(t-1) + q'(t-1), & t\in [1,1+\varepsilon]
\end{array}
\right. 
\end{equation}
\end{prop}
\begin{proof} The proof is similar to Proposition~\ref{prop_characteristic_simplest}.
We denote $F(z)=\int_{0}^T {\rm e}^{i z t} f(t)  \: {\rm d} t \in PW_{(0,T)}$ and $f(t)\in (E_T)^\perp$ if and only if there exists $Q(z)\in PW_{(0,\varepsilon)}$ such that $zQ(z)\in PW_{(0,\varepsilon)}$ and 
$
F(z)=D(z)Q(z).
$
By Paley-Wiener theorem there exists $q(t)\in L^2(0, \varepsilon)$ such that $Q(z)=\int_{0}^\varepsilon {\rm e}^{i z t} q(t) \: {\rm d} t$ and $q(t)$ is such that
$q(t)\in {W}^{1,2}(0,\varepsilon)$, $q(0)=q(\varepsilon)=0$.
Applying the inverse Fourier transform to $D(z)Q(z)$ we obtain
$$
\begin{array}{rcl}
f(t)& = & \frac{1}{2\pi}\int\limits_{-\infty}^{\infty} {\rm e}^{-i \xi t} D(\xi) Q(\xi) {\rm d} \xi \\
& = & \frac{1}{2\pi}\sum\limits_{k=0}^N a_k \int\limits_{-\infty}^{\infty} {\rm e}^{-i \xi (t-1+r_k)} Q(\xi) {\rm d} \xi 
- \frac{1}{2\pi}\sum\limits_{k=1}^N d_k  \int\limits_{-\infty}^{\infty} {\rm e}^{-i \xi (t-1+r_k)} i\xi Q(\xi) {\rm d} \xi\\
&  &- \frac{1}{2\pi}\int\limits_{-\infty}^{\infty} {\rm e}^{-i \xi (t-1)} i\xi Q(\xi) {\rm d} \xi 
\end{array}
$$
what proves the representation (\ref{eq_characteristic_eorth_retarded}).
\end{proof}

\begin{thm}\label{lem_characteristic_retarded}
For any initial condition $(y, x_0)\in \widehat{M}^2$ (or $(y, x_0)\in M^2$ in the case of retarded equation) and for any $T>1$ the minimum energy control of the equation (\ref{eq_intro_equation_neutral}) belongs to the characteristic space:
$$\widehat{u}_T(T-t; y,x_0)\in E_T.$$
\end{thm}
\begin{proof}
For a given $(y, x_0)$ and $T$ the optimal control $\widehat{u}(t)=\widehat{u}_T(T-t; y,x_0)$ is of the form (\ref{eq_admissible_control_retarded}) where $\widehat{u}_0(t)$ is given by (\ref{eq_min_controlu_retarded})--(\ref{eq_min_constantc_retarded}).
Thus for $t\in [1-r_{k},1-r_{k}+\varepsilon]$, $k\in\overline{1,N}$:
$$
\widehat{u}(T-t)=\varphi_k(1+\varepsilon-r_{k}-t) + d_k u(1+\varepsilon-r_{k}-t) + (d_k a_0 - a_k) \int\limits_0^{1+\varepsilon-r_{k}-t} {\rm e}^{a_0(1+\varepsilon-r_{k}-t-\tau)} \widehat{u}_0(\tau) \: {\rm d}\tau
$$
where $\varphi_k(\cdot)$ are given by (\ref{eq_admissible_control_neutral_varphi}) and $$\widehat{u}(T-t)=\widehat{u}_0(1+\varepsilon - t), \qquad t\in [1, 1+\varepsilon).$$
For an arbitrary $f(t)\in (E_T)^\perp$, we calculate the product $\langle \widehat{u}(T-t), f(t) \rangle_{L^2(0,T)}$. Taking into account the representation of $f(t)$ due to Proposition~\ref{prop_characteristic_neutral}, and applying changes of variables, we obtain:
 $$
 \begin{array}{l}
 \langle \widehat{u}(T-t), f(t) \rangle_{L^2(0,T)} =  
 \int\limits_0^{\varepsilon} \widehat{u}_0(\varepsilon - t)  (a_0 q(t) + q'(t)) \:{\rm d} t + \\
 +  \int\limits_{0}^{\varepsilon} \sum\limits_{k=1}^N  \left(\varphi_k(\varepsilon - t) + d_k u(\varepsilon - t) 
 + (d_k a_0 - a_k) \int\limits_0^{\varepsilon - t} {\rm e}^{a_0(\varepsilon - t-\tau)} \widehat{u}_0(\tau) \: {\rm d}\tau \right)  
 (a_k q(t) + d_k q'(t))\:{\rm d} t \\
 = a_0 \int\limits_0^{\varepsilon} \widehat{u}_0(t) q(\varepsilon - t) \:{\rm d} t
 + \int\limits_0^{\varepsilon} T_1(t) q'(\varepsilon - t) \:{\rm d} t
 + \int\limits_0^{\varepsilon} T_2(t) q(\varepsilon - t) \:{\rm d} t,
 \end{array}
 $$
where
$$
\begin{array}{l}
T_1(t) = \widehat{u}_0(t) +  \sum\limits_{k=1}^N \left[ d_k\varphi_k(t) + d_k^2 u(t) 
+ d_k(d_k a_0 - a_k) \int\limits_0^{t} {\rm e}^{a_0(t-\tau)} \widehat{u}_0(\tau) \: {\rm d}\tau \right],\\
T_2(t) = \sum\limits_{k=1}^N \left[ a_k\varphi_k(t) + d_k a_k u(t) 
+ a_k(d_k a_0 - a_k) \int\limits_0^{t} {\rm e}^{a_0(t-\tau)} \widehat{u}_0(\tau) \: {\rm d}\tau \right].
\end{array}
$$
Integrating by parts the term with $T_2(t)$ we get
$$
\int_0^{\varepsilon} T_2(t) q(\varepsilon - t) \:{\rm d} t = 
-\int_0^{\varepsilon} \left(\int_t^\varepsilon {\rm e}^{-a_0(t-\tau)} T_2(\tau) \: {\rm d}\tau \right) (a_0 q(\varepsilon - t) + q'(\varepsilon - t)) \:{\rm d} t.
$$
Further, 
$$
\begin{array}{l}
T_1(t) - \int\limits_t^\varepsilon {\rm e}^{-a_0(t-\tau)} T_2(\tau) \: {\rm d}\tau = 
G[\widehat{u}_0](t) 
-\sum\limits_{k=1}^N d_k^2 a_0 \int\limits_t^\varepsilon {\rm e}^{-a_0(t-\tau)} \widehat{u}_0(\tau) \: {\rm d}\tau\\
- \sum\limits_{k=1}^N d_k a_0 \int\limits_t^\varepsilon {\rm e}^{-a_0(t-\tau)} \varphi_k(\tau) \: {\rm d}\tau
- \sum\limits_{k=1}^N d_k a_0 (d_k a_0 - a_k)^2 \int\limits_t^\varepsilon {\rm e}^{-a_0(t-\tau)} \int\limits_0^{\tau} {\rm e}^{a_0(\tau-s)} \widehat{u}_0(s) \: {\rm d}s \: {\rm d}\tau,
 \end{array}
$$
and since the optimal control $\widehat{u}_0(t)$ satisfies the integral equation (\ref{eq_minl_thmeq_neutral0}): $G[\widehat{u}_0](t) = \alpha {\rm e}^{-a_0 t}$ we rewrite
$$
\begin{array}{l}
\int\limits_0^{\varepsilon} \left(T_1(t) - \int\limits_t^\varepsilon {\rm e}^{-a_0(t-\tau)} T_2(\tau) \: {\rm d}\tau\right) q'(\varepsilon - t) \:{\rm d} t = \\
a_0\int\limits_0^{\varepsilon} \left[-\alpha {\rm e}^{-a_0 t} 
+ \sum\limits_{k=1}^N \left(d_k^2  \widehat{u}_0(t) + d_k \varphi_k(t)
+ d_k^2 a_0 \int\limits_t^\varepsilon {\rm e}^{-a_0(t-\tau)} \widehat{u}_0(\tau) \: {\rm d}\tau
+ d_k (d_k a_0 - a_k) \int\limits_0^t {\rm e}^{a_0(t-\tau)} \widehat{u}_0(\tau) \: {\rm d}\tau
\right.\right.\\
\left.\left.
+ d_k a_0 (d_k a_0 - a_k) \int\limits_t^\varepsilon {\rm e}^{-a_0(t-\tau)} \int\limits_0^{\tau} {\rm e}^{a_0(\tau-s)} \widehat{u}_0(s) \: {\rm d}s \: {\rm d}\tau 
+  d_k a_0 \int\limits_t^\varepsilon {\rm e}^{-a_0(t-\tau)} \varphi_k(\tau) \: {\rm d}\tau
\right)\right] q(\varepsilon - t)
\:{\rm d} t
\end{array}
$$

Finally, collecting all the terms together we get
$$
 \langle \widehat{u}(T-t), f(t) \rangle_{L^2(0,T)}  
=a_0\int\limits_0^\varepsilon \left(G[\widehat{u}_0](t) - \alpha {\rm e}^{-a_0 t} \right) q(\varepsilon - t)  \:{\rm d} t = 0
$$
due to (\ref{eq_minl_thmeq_neutral0}), and thus $\widehat{u}(T-t)\in E_T$.
\end{proof}


\subsection*{Acknowledgments}

The author is warmly grateful to Yurii Lyubarskii for helpful discussions.


\bibliography{bib-storage}

\vskip7mm 
\noindent
{\bf Pavel Barkhayev} \\
Department of Mathematical Sciences, Norwegian University of Science and Technology, 7491, Trondheim, Norway  and \\
B.Verkin Institute for Low Temperature Physics and Engineering of the National Academy of Sciences of Ukraine, 47 Nauki Ave., 61103 Kharkiv, Ukraine \\
E-mail: {\em pavloba@ntnu.no}.

\end{document}